\documentclass[10pt]{amsart}

\usepackage{float}
\usepackage{physics}
\usepackage{amsrefs}

\usepackage{enumerate}
\usepackage{hyperref}
\hypersetup{
    colorlinks = true
}
\usepackage{amsfonts}
\usepackage{amssymb}
\usepackage{amsmath}
\usepackage{amsthm}
\usepackage[inline]{enumitem}
\usepackage{varwidth}

\usepackage{tikz-cd}
\usepackage{graphics}
\usepackage{graphicx}
\usepackage{color}

\usepackage{stmaryrd} 

\usepackage{xfrac} 

\usepackage[margin=1in]{geometry}

\newtheorem{theorem}{Theorem}
\newtheorem{corollary}[theorem]{Corollary}

\newtheorem{definition}[theorem]{Definition}
\newtheorem{lemma}[theorem]{Lemma}

\newtheorem{question}[theorem]{Question}

\newtheorem{remark}[theorem]{Remark}

\def\s{\subseteq}

\newtheorem*{theoremstar}{Theorem}

\usepackage{todonotes}

\title{Approximation properties of torsion classes}
\author[Cox]{Sean Cox}
\address[Cox]{Department of Mathematics and Applied Mathematics, Virginia Commonwealth
University, 1015 Floyd Avenue, Richmond, Virginia 23284, USA}
\email{scox9@vcu.edu}
\author[Poveda]{Alejandro Poveda}
\address[Poveda]{ Department of Mathematics and Center of Mathematical Sciences and Applications, Harvard University, Cambridge MA, 02138, USA}
\email{alejandro@cmsa.fas.harvard.edu}
\author[Trlifaj]{Jan Trlifaj}
\address[Trlifaj]{Faculty of Mathematics and Physics, Department of Algebra, Charles University,  Sokolovsk\'a 83,
186 75 Prague 8, Czech Republic}
\email{trlifaj@karlin.mff.cuni.cz}

\thanks{ The first author
was supported by the National Science Foundation (DMS-2154141). The second author acknowledges support from the Department of Mathematics at Harvard University as well as from the Harvard Center of Mathematical Sciences and Applications (CMSA). The research of the third author was supported by GA\v{C}R 23-05148S}


\begin{document}

\begin{abstract}
We strengthen a result of Bagaria and Magidor~\cite{MR3152715} about the relationship between large cardinals and torsion classes of abelian groups, and prove that 
\begin{enumerate*}
    \item the \emph{Maximum Deconstructibility} principle introduced in \cite{Cox_MaxDecon} requires large cardinals; it sits, implication-wise, between Vop\v{e}nka's Principle and the existence of an $\omega_1$-strongly compact cardinal. 
    \item While deconstructibility of a class of modules always implies the precovering property by \cite{MR2822215}, the concepts are (consistently) non-equivalent, even for classes of abelian groups closed under extensions, homomorphic images, and colimits.  
\end{enumerate*}
\end{abstract}

\maketitle
\tableofcontents

\section{Introduction}

Approximation theory of modules provides tools for studying general modules by choosing special classes of modules and finding approximations (\emph{precovers} and \emph{preenvelopes)}, and minimal approximations (\emph{covers} and \emph{envelopes}), of the general modules by modules in the chosen classes.

A model case is the special class $\mathcal P_0$ of \emph{projective modules} that yields {projective precovers}, and projective resolutions of modules. Similarly, the class $\mathcal F_0$ of \emph{flat modules} yields flat covers, and minimal flat resolutions of modules, \cite{MR1832549}. Dually, the class $\mathcal I_0$ of \emph{injective modules} yields injective envelopes, and minimal injective coresolutions of modules.

In the case when minimal approximations exist, they provide invariants of modules, the model examples being Bass’ invariants and dual Bass’ invariants of modules over commutative noetherian rings \cite{MR1753146}.

Approximation theory makes it possible to extend classic homological algebra to the setting of relative homological algebra \cite{EilMoo}. Applications of the extended setting are numerous: in the theory of finitely generated modules over artin algebras \cite{AR}, tilting theory of commutative noetherian rings \cite{MR2985554}, and Gorenstein homological algebra \cite{MR3839274}, to name a few.

\medskip

Classic structure theory of modules is based on decompositions into (possibly infinite) direct sums of indecomposable modules. This is a strong tool in the study of modules of finite length and injective modules over commutative noetherian rings \cite{AF}. However, decomposability fails for many other classes of interest. In contrast, the weaker notion of \emph{deconstructibility} (namely, decomposition into a transfinite extension of modules from a given \emph{set} of modules) is almost omnipresent.

The key point is that deconstructible classes of modules always provide precovers, that is, each deconstructible class is precovering \cite{MR2822215}. The opposite question is more subtle:

\begin{question}\label{q_DoesPrecoverImplyDecon}
Suppose $\mathcal C$ is a class of modules that is precovering and closed under transfinite extensions. Is $\mathcal C$ deconstructible?
\end{question}
In this paper we use some results of Bagaria and Magidor from \cite{MR3152715} to provide a negative answer to Question \ref{q_DoesPrecoverImplyDecon}, at least in the absence of large cardinals:
\begin{theoremstar}
If there are no $\omega_1$-strongly compact cardinals, then there is a precovering class of abelian groups that is closed under transfinite extensions, homomorphic images, and colimits, but is \emph{not} deconstructible.
\end{theoremstar}

In fact, the torsion class 
\[
{}^{\perp_0} \mathbb{Z}:= \{ A \in \textbf{Ab} \ : \ \text{Hom}(A,\mathbb{Z})=0 \}
\]
is always a covering class that is closed under transfinite extensions, homomorphic images, and colimits.  In this respect we prove:

\begin{theorem}\label{thm_StrengthenBM}
${}^{\perp_0} \mathbb{Z}$ is deconstructible if and only if there is an $\omega_1$-strongly compact cardinal.    
\end{theorem}
Theorem \ref{thm_StrengthenBM} strengthens a result of Bagaria and Magidor, who got the same equivalence with ``deconstructible" weakened to ``bounded" (we show in Lemma \ref{lemma: key lemma} that these concepts are equivalent, for any class of modules closed under transfinite extensions, homomorphic images, and colimits).

In \cite{Cox_MaxDecon}, the first author proved that any deconstructible class is:
\begin{enumerate}[label=(\alph*)]
    \item\label{item_ClosedTFExt} closed under transfinite extensions (by definition), and
    \item\label{item_AE_quotients} ``eventually almost everywhere closed under quotients".  This is an extremely weak version of saying that if $A \subset B$ are both in the class, then so is $B/A$.  In particular, it holds if the class is closed under homomorphic images.
\end{enumerate}

He introduced the \textbf{Maximum Deconstructiblity} principle, which asserts that any class satisfying both  \ref{item_ClosedTFExt} and \ref{item_AE_quotients} is deconstructible, and proved that 
\begin{center}
    Vop\v{e}nka's Principle (VP) implies Maximum Deconstructibility.
\end{center}
Maximum Deconstructibility appeared very powerful, since \cite{Cox_MaxDecon} showed that it implied deconstructibility of many classes in Gorenstein Homological Algebra that (so far) are not known to be deconstructible in ZFC alone.  But it was unclear whether Maximum Deconstructibility had any large cardinal strength at all.  We again use the  results of Bagaria and Magidor to show that it does:
\begin{theoremstar}
    Maximum Deconstructibility implies the existence of an $\omega_1$-strongly compact cardinal.
\end{theoremstar}

So Maximum Deconstructibility lies, implication-wise, between VP and the existence of an $\omega_1$-strongly compact cardinal.  The first author still conjectures that it is equivalent to VP.

Section \ref{sec_Prelims} has preliminaries.  Section \ref{sec_ApproxZFC} proves the approximation properties of torsion classes that are provable in ZFC alone.  Section \ref{sec_Clarify_BM} discusses some results of Bagaria-Magidor~\cite{MR3152715}.  Section \ref{sec: max decon needs large cardinals} proves the main theorems mentioned above, and Section \ref{sec_Questions} includes some questions.

\section{Preliminaries}\label{sec_Prelims}

Our notation and conventions follow Kanamori~\cite{MR1994835} (for set theory and large cardinals) and G\"obel-Trlifaj~\cite{MR2985554} (for module theory).  By \emph{ring} we will mean a unital and not necessarily commutative ring.  If $R$ is a ring, the class of left $R$-modules will be denoted by $R$-Mod.  We will say that $M$ is a \emph{module} rather than a  left $R$-module whenever $R$ is clear from the context. For a regular cardinal $\kappa$ and a module $M$, the collection of all submodules $N$ of $M$ that are ${<}\kappa$-generated will be denoted by $[M]^{<\kappa}$; if $|R|<\kappa$ and $\kappa$ is regular and uncountable, then ``${<}\kappa$-generated" is equivalent to ``of cardinality less than $\kappa$".  A cardinal $\kappa$ is \emph{strongly compact} if every $\kappa$-complete filter (on any set whatsoever) can be extended to a $\kappa$-complete ultrafilter.   Bagaria and Magidor considered a weakening of strong compactness:
\begin{definition}[Bagaria and Magidor, {\cite{MR3152715}}]\label{def: omega1strongly}
    An uncountable cardinal $\kappa$ is called \emph{$\omega_1$-strongly compact} if every $\kappa$-complete filter extends to an $\omega_1$-complete ultrafilter.
\end{definition}
Note that if $\kappa$ is $\omega_1$-strong compact then so is any cardinal $\lambda\geq \kappa$.  Unlike strongly compact cardinals, an $\omega_1$-strong compact cardinal  may fail to be regular \cite[\S6]{MR3152715}.

\section{Approximation properties and torsion classes}\label{sec_ApproxZFC}

Given a class $\mathcal{K}$ of $R$-modules, a \emph{$\mathcal{K}$-filtration} is a $\subseteq$-increasing and $\subseteq$-continuous sequence $\langle M_\xi\mid \xi<\eta\rangle$ of modules such that $M_0=0$ and for all $\xi<\eta$ such that $\xi+1<\eta$, $M_\xi$ is a submodule of $M_{\xi+1}$ and $M_{\xi+1}/M_\xi$ is isomorphic to a member of $\mathcal{K}$. A module $M$ is \emph{$\mathcal{K}$-filtered} whenever there is a $\mathcal{K}$-filtration $\langle M_\xi\mid \xi<\eta\rangle$ whose union is $M$. We shall denote the class of all $\mathcal{K}$-filtered modules by $\mathrm{Filt}(\mathcal{K})$. $\mathcal{K}$ is  \emph{closed under transfinite extensions} whenever $\mathrm{Filt}(\mathcal{K})\subseteq\mathcal{K}$.  Finally, 
$\mathcal{K}^{<\kappa}$ denotes the class of all ${<}\kappa$-presented members of $\mathcal{K}$.

\begin{definition}\label{def: decomposability}
   Let  $\mathcal{K}$ be a class of modules and $\lambda\leq \kappa$ be regular cardinals.
    \begin{enumerate}
        \item   $\mathcal{K}$ is \emph{$\kappa$-deconstructible} whenever $\mathcal{K} = \text{Filt}\left( {\mathcal{K}}^{<\kappa} \right)$; equivalently, $\mathcal{K}$ is closed under transfinite extensions, and  every $M\in\mathcal{K}$ admits a ${\mathcal{K}}^{<\kappa}$-filtration.\footnote{This is the definition of deconstructibility in most newer references, such as \cite{sar_trl_DeconAEC}, since it is the version that (by \cite{MR2822215}) implies the precovering property.  Some other sources (e.g. Cox~\cite{Cox_MaxDecon}, G\"obel-Trlifaj~\cite{MR2985554}) only require that $\mathcal{K} \subseteq \text{Filt}\left( {\mathcal{K}}^{<\kappa} \right)$ (i.e., with $\subseteq$ rather than equality).}

        \item $\mathcal{K}$ is \emph{$(\kappa,\lambda)$-cofinal} if  $\mathcal{K}\cap [M]^{<\kappa}$ is $\s$-cofinal in $[M]^{<\lambda}$ for every module $M\in\mathcal{K}$ (i.e., whenever every $N\in [M]^{<\lambda}$ is contained in some ${<}\kappa$-generated submodule of $M$ that lies in $\mathcal{K}$).
        \item $\mathcal{K}$ is \emph{bounded by $\kappa$} whenever $M=\sum \left( \mathcal{K}\cap [M]^{<\kappa} \right)$ for all $M\in \mathcal{K}$ (i.e., if every $x \in M$ is contained in some ${<}\kappa$-generated submodule of $M$ that lies in $\mathcal{K}$).\footnote{This property was introduced by Gardner in \cite{MR0647168} and later investigated by Dugas in \cite{MR0799895}.}
         \item  $\mathcal{K}$ is  \emph{$\kappa$-decomposable} if  every module in $\mathcal{K}$ is a direct sum of ${<}\kappa$-presented modules from $\mathcal{K}$.
    \end{enumerate}
   $\mathcal{K}$ is said to be \emph{deconstructible} provided  it is $\kappa$-deconstructible for a regular cardinal $\kappa$. The same convention is applied to the rest of above-mentioned properties.
\end{definition}
Clearly any $(\kappa,\lambda)$-cofinal class is bounded by $\kappa$. In the forthcoming Lemma~\ref{lemma: key lemma} we will argue that for certain classes $\mathcal{K}$,  ``bounded by $\kappa$" and $``(\kappa,\kappa)$-cofinal" are equivalent concepts.

The following definitions are due to Enochs (generalizing earlier work of Auslander):
\begin{definition}
Let $\mathcal{K}$ be a class of modules. We say that $\mathcal{K}$ is:
\begin{enumerate}
    \item \emph{Precovering}: If every module $M$ posseses a \emph{$\mathcal{K}$-precover}; namely, a morphism $f\colon C\rightarrow M$ with $C\in \mathcal{K}$ such that $\mathrm{Hom}(D,f)$ is surjective for all $D\in\mathcal{K}.$ 
    \item \emph{Covering}: If every module $M$ posseses a \emph{$\mathcal{K}$-cover}; namely, a $\mathcal{K}$-precover $f\colon C\rightarrow M$ such that for each  $g\in\mathrm{End}(C)$ if $fg=f$ then $g$ is an automorphism of $C$.
\end{enumerate}

\end{definition}

Saor\'{\i}n and \v{S}\v{t}ov\'{\i}\v{c}ek ~\cite{MR2822215} proved that if a class $\mathcal{K}$ (of, say, modules) is deconstructible, then it is precovering.  In \S\ref{sec: max decon needs large cardinals}  we argue that there are nicely-behaved covering classes which are yet, at least consistently, not deconstructible. The concrete example we have in mind is the torsion class of the abelian group $\mathbb{Z}$ in a context where the set-theoretic universe does not have any $\omega_1$-strongly compact cardinals (see Definition~\ref{def: omega1strongly}). 

Dickson~\cite{MR0191935} introduced the concept of \emph{torsion pairs}, which are pairs $(\mathcal{A},\mathcal{B})$ such that
\[
\mathcal{A} = {}^{\perp_0} \mathcal{B}:= \{ X \ : \ \text{Hom}(X,B)=0 \text{ for all } B \in \mathcal{B} \}
\]
and
\[
\mathcal{B} = \mathcal{A}^{\perp_0}:= \{ Y \ : \ \text{Hom}(A,Y)=0 \text{ for all } A \in \mathcal{A} \}.  
\]
A class of modules is called a \emph{torsion class} if it is of the form ${}^{\perp_0} \mathcal{B}$ for some class $\mathcal{B}$; notice that in this situation, 
\[
\left(  {}^{\perp_0} \mathcal{B}, \big( {}^{\perp_0} \mathcal{B} \big)^{\perp_0}  \right)
\]
is easily seen to be a torsion pair (the torsion pair \emph{cogenerated by} $\mathcal{B}$).  So torsion classes are exactly those classes that are the left part of some torsion pair.  Whenever $\mathcal{X}$ is a singleton $\{ X \}$, the convention is to write ${}^{\perp_0} X$ rather than ${}^{\perp_0}\{X\}$.

By \cite{MR0191935}, torsion classes are exactly those classes that are closed under arbitrary direct sums, extensions, and homomorphic images (\cite{MR0389953}, Proposition VI.2.1).  They have many other desirable features; we list the ones that are relevant for this paper:

\begin{lemma}\label{lem_NicePropertiesTorsionClasses}
Torsion classes are:
\begin{enumerate}
    \item\label{item_ClosedUnderImages} closed under homomorphic images

    \item\label{item_ClosedUnderColimits} closed under colimits
    \item\label{item_ClosedUnderTransExt} closed under transfinite extensions
    \item\label{item_TorsionClassesCovering} covering classes.
\end{enumerate}
\end{lemma}

\eqref{item_ClosedUnderImages} is immediate from the definition of torsion class, and  \eqref{item_ClosedUnderColimits} follows from closure under direct sums and homomorphic images.  Closure under transfinite extensions follows easily from closure under extensions and colimits.  Finally, given any module $M$ and any torsion class $\mathcal{T}:={}^{\perp_0} \mathcal{X}$, the \emph{trace of $\mathcal{T}$ in $M$} is defined as
\[
r(M):= \sum_{T \in \mathcal{T}} \{ \text{im} \  \pi \ : \ \pi \in \text{Hom}(T,M)  \}.
\]
Then it is easily seen that $r(M) \in \mathcal{T}$ and that the inclusion $r(M) \to M$ is a $\mathcal{T}$-cover of $M$.

\section{On some results of Bagaria and Magidor}\label{sec_Clarify_BM}

In Bagaria-Magidor~\cite[\S5]{MR3152715} a torsion class ${}^{\perp_0} \mathcal{X}$ is called ``$\kappa$-generated" whenever every $A \in {}^{\perp_0} \mathcal{X}$ is a \emph{direct} sum of subgroups in ${}^{\perp_0} \mathcal{X}$, each of cardinality ${<}\kappa$; this is what is usually called $\kappa$-decomposable (see Definition~\ref{def: decomposability}).  The use of the word ``direct" on page 1867 of \cite{MR3152715} appears to be a misprint,\footnote{The anonymous referee has suggested that the authors of \cite{MR3152715} might have intended to say \emph{directed sum} instead of \emph{direct sum}.} since the arguments there never use (or conclude) directness of the relevant sums. The arguments in \cite[\S5]{MR3152715}, and the citation they provide for the concept (Dugas~\cite{MR0799895}), appear to use the weaker property that every group in the class is a sum of subgroups in the class of cardinality ${<}\kappa$.  The latter condition is what we called \emph{bounded by $\kappa$} in Section \ref{sec_Prelims} (this terminology was used by Gardner~\cite{MR0647168} and Dugas~\cite{MR0799895}).  So we shall use the \emph{bounded by $\kappa$} terminology to state the Bagaria-Magidor results.  They proved:

\begin{theorem}[{\cite[Theorem 5.1]{MR3152715}}]\label{thm_BM_5_1_corrected}
If $\kappa$ is a $\delta$-strongly compact cardinal and $X$ is an abelian group of cardinality less than $\delta$, then ${}^{\perp_0} X$ is bounded by $\kappa$.  In fact, they prove the stronger property of ${}^{\perp_0} X$ being $(\kappa,\omega_1)$-cofinal.\footnote{Recall that this stands for the following property:  For each $G\in{}^{\perp_0}\mathbb{Z}$ every countable subgroup $H\leq G$ is included in a member of ${}^{\perp_0}\mathbb{Z}\cap [G]^{<\kappa}$ (see Definition~\ref{def: decomposability}). }
\end{theorem}

\begin{theorem}[{\cite[Theorem 5.3]{MR3152715}}]\label{thm_BM_5_3_corrected}
If ${}^{\perp_0} \mathbb{Z}$ is bounded in $\kappa$, then $\kappa$ is $\omega_1$-strongly compact.
\end{theorem}

\begin{corollary}[{\cite[Corollary 5.4]{MR3152715}}]\label{cor_BoundedIFFomega_1sc}
${}^{\perp_0} \mathbb{Z}$ is bounded in $\kappa$ if and only if $\kappa$ is $\omega_1$-strongly compact.    
\end{corollary}

The following ZFC theorem (due to the third author) shows that, while ${}^{\perp_0} \mathbb{Z}$ can be bounded (by the Bagaria-Magidor results), it can never be decomposable:
\begin{theorem}\label{thm_TorsionNotDecomp}
${}^{\perp_0} \mathbb{Z}$ is not decomposable.
\end{theorem}
\begin{proof}
Suppose toward a contradiction that ${}^{\perp_0} \mathbb{Z}$ is $\kappa$-decomposable, with $\kappa$ (without loss of generality) a regular uncountable cardinal.  Fix any prime $p$.  Then all $p$-groups\footnote{Abelian groups whose elements all have order a power of $p$.} are in ${}^{\perp_0} \mathbb{Z}$.  So, in particular, the $\kappa$-decomposability assumption of ${}^{\perp_0} \mathbb{Z}$ implies
\begin{equation}\label{eq_p_group_decomp}
    \text{Every } p \text{-group is a direct sum of } < \kappa \text{-sized subgroups.}
\end{equation}

For a $p$-group $G$, we can consider its $p$-length, which is the least ordinal $\sigma$ such that $G_\sigma = G_{\sigma+1}$, where $G_\sigma$ is defined recursively by $G_0:=G$, $G_{\sigma+1} = \{ p g \ : \ g \in G_\sigma \}$, and $G_\sigma = \bigcap_{\xi < \sigma} G_\xi$ for limit ordinals $\sigma$.  By a construction of Walker~\cite{MR0364497} (see also Bazzoni-\v{S}\v{t}ov\'{\i}\v{c}ek~\cite{MR3011880}), there exists a $p$-group, denoted $P_{\kappa^+}$, whose $p$-length is exactly $\kappa^+ +1$.\footnote{The group is indexed by strictly decreasing finite sequences of ordinals $\kappa^+ > \beta_1 > . . . > \beta_k$, with relations $p(\kappa^+\beta_1 . . . \beta_{k+1}) =
\kappa^+\beta_1 . . . \beta_k$ and $p(\kappa^+) = 0$.}  By \eqref{eq_p_group_decomp}, 
\[
P_{\kappa^+} = \bigoplus_{i \in I} Q_i
\]
for some collection $(Q_i)_{i \in I}$ of $<\kappa$-sized subgroups. 
 Since subgroups of $p$-groups are also $p$-groups, each $Q_i$ has a $p$-length, which is $<\kappa$ because $|Q_i|<\kappa$.  And it is easy to check that the $p$-length of a direct sum is at most the supremum of the $p$-lengths of the direct summands, so $\bigoplus_{i \in I} Q_i$ has $p$-length at most $\kappa$, contradicting that the $p$-length of $P_{\kappa^+}$ is $\kappa^+ + 1$.
\end{proof}

\section{Maximum Deconstructibility, and deconstructibility of torsion classes, require large cardinals}\label{sec: max decon needs large cardinals}

Through this section $\kappa$ will denote a  regular uncountable cardinal. Our main goal is to clarify, for a given class of modules $\mathcal{K}$, the  relationship between $\mathcal{K}$ being $\kappa$-deconstructible and $\mathcal{K}$ being bounded in $\kappa$. As noted in \S\ref{sec_Prelims}, any $(\kappa,\kappa)$-cofinal class $\mathcal{K}$ is bounded in $\kappa$.  In fact the former property seems to be  strictly stronger than the latter, and they both follow from $\kappa$-deconstructibility.  Lemma \ref{lemma: key lemma} provides some important scenarios where these concepts are equivalent.  Combining this lemma with the  results of Bagaria and Magidor from Section \ref{sec_Clarify_BM} allows us to prove the main results mentioned in the introduction, namely:

\begin{enumerate}
    \item  that the \emph{Maximum Deconstructibility} principle introduced in \cite{Cox_MaxDecon} entails the existence of large cardinals;

    \item that in the absence of $\omega_1$-strongly compact cardinals, the class ${}^{\perp_0} \mathbb{Z}$ is a covering class (cf. Lemma \ref{lem_NicePropertiesTorsionClasses} above) that is \emph{not} deconstructible.

\end{enumerate}

\begin{lemma}\label{lemma: key lemma}
  Suppose $\kappa$ is a regular uncountable cardinal and $R$ is a ring of size less than $\kappa$, and $\mathcal{K}$ is a class of $R$-modules. Consider the following statements:
    \begin{enumerate}[label=(\alph*)]
        \item\label{item_K_kappa_decon} $\mathcal{K}$ is $\kappa$-deconstructible.
        \item\label{item_KappaCofinal} $\mathcal{K}$ is $(\kappa,\kappa)$-cofinal and closed under transfinite extensions.

        \item\label{item_KappaBounded} $\mathcal{K}$ is bounded by $\kappa$ and closed under transfinite extensions.
        
    \end{enumerate}

Then:
\begin{enumerate}
    \item \ref{item_K_kappa_decon} $\implies$ \ref{item_KappaCofinal} $\implies$ \ref{item_KappaBounded}.

    \item\label{item_AdditionalAssumeQuot} If $\mathcal{K}$ is closed under quotients---i.e., if 
    \[
     \big(A \subset B, \ A \in \mathcal{K}, \text{ and } B \in \mathcal{K} \big) \implies \ B/A \in \mathcal{K}
    \]
    for all modules $A$ and $B$---then \ref{item_K_kappa_decon} and \ref{item_KappaCofinal} are equivalent.  

   \item If $\mathcal{K}$ is closed under homomorphic images and colimits, then all three statements are equivalent.
\end{enumerate}
\end{lemma}

\begin{proof}
The \ref{item_K_kappa_decon} $\Rightarrow$ \ref{item_KappaCofinal} direction follows from the Hill Lemma (cf.\ G\"obel-Trlifaj~\cite{MR2985554}, Theorem 7.10).  The \ref{item_KappaCofinal} $\Rightarrow$ \ref{item_KappaBounded} implication is immediate from the definitions.
    
Now suppose $\mathcal{K}$ is closed under transfinite extensions and quotients.  We prove \ref{item_KappaCofinal} $\Rightarrow$ \ref{item_K_kappa_decon}.  Without losing any generality we may assume that (the universe of) every module $M\in\mathcal{K}$ is a cardinal. Since $\mathcal{K}$ is closed both under transfinite extensions and  quotients, and $|R|<\kappa$, then by \cite[Theorem~6.3]{Cox_MaxDecon} in order to prove $\kappa$-deconstructibility of $\mathcal{K}$ it suffices to show that if $M \in \mathcal{K}$, then $\mathcal{K} \cap \wp^*_\kappa(M)$ is stationary in $\wp^*_\kappa(M)$, where $\wp^*_\kappa(M)$ denotes the ${<}\kappa$-sized submodules of $M$ whose intersection with $\kappa$ is transitive (i.e., those $<\kappa$-sized submodules $N$ of $M$ such that $N\cap \kappa\in \kappa$).

    The set $\mathcal{K}\cap \wp^*_\kappa(M)$ being stationary in $\wp^*_\kappa(M)$ is equivalent to the following statement (\cite{MattHandbook}):  for each function $F$ from finite subsets of $M$ into $M$, there is an $X \in \mathcal{K}\cap \wp^*_\kappa(M)$ such that $X$ is closed under the function $F$.  We prove this next.

    Fix an arbitrary function $F$ from \emph{finite} subsets of $M$ into $M$; notice that since $F$ is finitary, for any infinite $D \subset M$, the closure of $D$ under the function $F$ has the same cardinality as $D$.  Using the assumption that $\mathcal{K}$ is $\subseteq$-cofinal in $[M]^{<\kappa}$, fix some $C_0 \in \mathcal{K} \cap [M]^{<\kappa}$.  Recursively build a $\subseteq$-increasing $\omega$-chain
    \[
    C_0 \subseteq X_0 \subseteq C_1 \subseteq X_1 \subseteq C_2 \subseteq X_2 \subseteq \dots
    \]
    such that for each $n \in \omega$:
    \begin{itemize}
        \item $X_n$ is the closure of $C_n \cup \text{sup}(C_n \cap \kappa)$ under the function $F$.  This closure has cardinality $<\kappa$, because $|C_n|<\kappa$ and hence (by regularity of $\kappa$)  $|C_n \cup \text{sup}(C_n \cap \kappa)| < \kappa$.

        \item (for $n \ge 1$) $C_n \in \mathcal{K}$, $C_n \supseteq X_{n-1}$, and $|C_n|<\kappa$.  This is possible by the assumption that $\mathcal{K}$ is $\subseteq$-cofinal in $[M]^{<\kappa}$ (but note that $C_n \cap \kappa$ might fail to be transitive).
    \end{itemize}
    Let $X:= \bigcup_{n<\omega} X_{n}$.  Then $X$ is closed under $F$, $|X|<\kappa$, and $X \cap \kappa=\text{sup}_{n < \omega} \text{sup}(C_n \cap \kappa)$; in particular, $X \cap \kappa$ is transitive. Also, note that $X = \bigcup_{n<\omega} C_n$.  Since $\mathcal{K}$ is closed under  quotients, each $C_{n+1}/C_n$ is in $\mathcal{K}$.   So since $\mathcal{K}$ is closed under transfinite extensions, and each $C_{n+1}/C_n$ is in $\mathcal{K}$, we conclude that $X \in \mathcal{K}$.

Now suppose that $\mathcal{K}$ is closed under homomorphic images and colimits; we show that \ref{item_KappaBounded} implies \ref{item_KappaCofinal} (and hence \ref{item_K_kappa_decon} will hold too, since closure under homomorphic images trivially implies closure under quotients in the sense of \ref{item_AdditionalAssumeQuot}).  Suppose $\mathcal{K}$ is bounded in $\kappa$.  Suppose $M \in \mathcal{K}$, and fix any $X \in [M]^{<\kappa}$.  By $\kappa$-boundedness of $\mathcal{K}$, for each $x \in X$ there is a $Y_x \in \mathcal{K} \cap [M]^{<\kappa}$ with $x \in Y_x$.  Then $S:=\sum_{x \in X} Y_x$ is a $<\kappa$-sized submodule of $M$ and contains $X$.  And $S$ is a homomorphic image of the colimit of the (possibly non-directed) diagram of inclusions
\[
\big\{ Y_x \to Y_z \ : \ Y_x \subseteq Y_z \text{ and } x,z \in X  \big\}.
\]
So by closure of $\mathcal{K}$ under homomorphic images and colimits, $S \in \mathcal{K}$.
\end{proof}

We now focus on abelian groups. Recall that by  Lemma~\ref{lem_NicePropertiesTorsionClasses} any torsion class ${}^{\perp_0} \mathcal{X}$ is closed under homomorphic images, transfinite extensions, and colimits.  So by Lemma \ref{lemma: key lemma}, a class of the form ${}^{\perp_0} \mathcal{X}$ is deconstructible if and only if it is bounded in some cardinal.  Furthermore, as noted in the introduction, although an $\omega_1$-strongly compact cardinal might be singular, all cardinals above it are also $\omega_1$-strongly compact.  In particular (by considering its successor if necessary), there exists an $\omega_1$-strongly compact cardinal if and only if there exists a \emph{regular} $\omega_1$-strongly compact cardinal.  Then Lemma \ref{lem_NicePropertiesTorsionClasses}, Lemma \ref{lemma: key lemma}, and the  Corollary \ref{cor_BoundedIFFomega_1sc} of Bagaria and Magidor immediately imply Theorem \ref{thm_StrengthenBM} from the introduction, which asserted that ${}^{\perp_0} \mathbb{Z}$ is deconstructible if and only if there is an $\omega_1$-strongly compact cardinal.  In fact, making use of the  Theorem \ref{thm_BM_5_1_corrected} of Bagaria and Magidor, we have:

\begin{corollary}\label{cor_MainCor}
The following are equivalent:
\begin{enumerate}
    \item  There exists an $\omega_1$-strongly compact cardinal.
    \item ${}^{\perp_0} \mathbb{Z}$ is deconstructible.
    \item For all countable abelian groups $X$, ${}^{\perp_0} X$ is deconstructible.
\end{enumerate}
\end{corollary}

Since ${}^{\perp_0} \mathbb{Z}$ is closed under transfinite extensions and homomorphic images, we solve part of \cite[Question~8.2]{Cox_MaxDecon} in the affirmative:
\begin{corollary}\label{cor_ImprovedBagMag}
   \emph{Maximum Deconstructibility} of \cite{Cox_MaxDecon} has large cardinal consistency strength. Specifically, it lies implication-wise, in between the existence of an $\omega_1$-strongly compact and Vop\v{e}nka's Principle. 
\end{corollary}

Recall that  deconstructible classes of modules are always precovering (\cite[Theorem~7.21]{MR2985554}). Thus, another interesting corollary of Corollary \ref{cor_MainCor} and Lemma~\ref{lem_NicePropertiesTorsionClasses} is the existence of precovering classes (even covering classes) which are not deconstructible:
\begin{corollary}
If there are no $\omega_1$-strongly compact cardinals, then there is a covering class in \textbf{Ab} that is closed under colimits, transfinite extensions, and homomorphic images, but is not deconstructible (namely, the class ${}^{\perp_0} \mathbb{Z}$). Thus, $\mathrm{ZFC}$ cannot prove that all subclasses of \textbf{Ab} satisfying the clauses in Lemma \ref{lem_NicePropertiesTorsionClasses} are deconstructible.\qed
\end{corollary}

\begin{remark}
    The following statement has recently been proven in \cite[Remark 4 and Corollary 3.4]{PPT}: Assume Vop\v{e}nka’s Principle holds. Then for any ring $R$, each class of $R$-modules that is closed under finite direct sums, extensions, and direct limits is deconstructible.
\end{remark}

\section{Questions}\label{sec_Questions}

\begin{question}
    Suppose that $\kappa\geq \omega_2$ is a regular cardinal such that ${}^{\perp_0}X$ is deconstructible for all abelian groups of size ${<}\kappa$. Must $\kappa$ be strongly compact?
\end{question}

Much of the literature focuses on deconstructibility and precovering properties of ``roots of Ext" classes, i.e., classes of the form
\[
{}^\perp \mathcal{B}:= \big\{ A \ : \ \forall B \in \mathcal{B} \ \text{Ext}^1(A,B)=0 \big\}
\]
for some fixed class $\mathcal{B}$ of modules.  Our results in Section \ref{sec: max decon needs large cardinals} consistently separates deconstructibility from precovering for ${}^{\perp_0} \mathbb{Z}$, but this is very far from being a root of Ext; it does not even contain the ring $\mathbb{Z}$.  

This suggests asking whether deconstructibility is equivalent to precovering for classes of the form ${}^\perp \mathcal{B}$, but there is a caveat: the first author proved in \cite{Cox_VP_CotPairs} that it is consistent, relative to consistency of Vop\v{e}nka's Principle (VP), that over every \emph{hereditary} ring (such as $\mathbb{Z}$), \emph{every} class of the form ${}^\perp \mathcal{B}$ is deconstructible.  So deconstructibility is trivially equivalent to precovering for roots of Ext (over hereditary rings) in that model.  But the following questions are open:

\begin{question}\label{q_ConPrecoverNotDecon_Ab}
Is it consistent with ZFC that there is a class $\mathcal{B}$ of abelian groups such that ${}^\perp \mathcal{B}$ is precovering, but not deconstructible?
\end{question}

\begin{question}
Does ZFC prove the existence of a ring $R$ and a class $\mathcal{B}$ of $R$-modules, such that ${}^\perp \mathcal{B}$ is precovering, but not deconstructible?  By the remarks above, such a ring could not be provably hereditary (unless VP is inconsistent).
\end{question}

Both questions are even open if we replace ``classes of the form ${}^\perp \mathcal{B}$" with ``classes that are closed under transfinite extensions and contain the ring" (such classes would contain all free modules).


\bib{AF}{book}{
  author={Anderson, Frank W.},
  author={Fuller, Kent R.},
  title={Rings and categories of modules},
  series={Graduate Texts in Mathematics},
  volume={13},
  edition={2},
  publisher={Springer-Verlag, New York},
  date={1992},
  pages={x+376},
}

\bib{AR}{article}{
  author={Auslander, Maurice},
  author={Reiten, Idun},
  title={Applications of contravariantly finite subcategories},
  journal={Adv. Math.},
  volume={86},
  date={1991},
  number={1},
  pages={111--152},
  issn={0001-8708},
  review={\MR {1097029}},
  doi={10.1016/0001-8708(91)90037-8},
}

\bib{MR3152715}{article}{
  author={Bagaria, Joan},
  author={Magidor, Menachem},
  title={Group radicals and strongly compact cardinals},
  journal={Trans. Amer. Math. Soc.},
  volume={366},
  date={2014},
  number={4},
  pages={1857--1877},
  issn={0002-9947},
}

\bib{MR3011880}{article}{
  author={Bazzoni, Silvana},
  author={\v {S}\v {t}ov\'{\i }\v {c}ek, Jan},
  title={On the abelianization of derived categories and a negative solution to Rosick\'{y}'s problem},
  journal={Compos. Math.},
  volume={149},
  date={2013},
  number={1},
  pages={125--147},
  issn={0010-437X},
}

\bib{MR1832549}{article}{
  author={Bican, L.},
  author={El Bashir, R.},
  author={Enochs, E.},
  title={All modules have flat covers},
  journal={Bull. London Math. Soc.},
  volume={33},
  date={2001},
  number={4},
  pages={385--390},
  issn={0024-6093},
}

\bib{Cox_MaxDecon}{article}{
  author={Cox, Sean},
  title={Maximum deconstructibility in module categories},
  journal={Journal of Pure and Applied Algebra},
  volume={226},
  number={5},
  year={2022},
}

\bib{Cox_VP_CotPairs}{article}{
  author={Cox, Sean},
  title={Salce's problem on cotorsion pairs is undecidable},
  journal={Bull. Lond. Math. Soc.},
  volume={54},
  date={2022},
  number={4},
  pages={1363--1374},
}

\bib{MR0191935}{article}{
  author={Dickson, Spencer E.},
  title={A torsion theory for Abelian categories},
  journal={Trans. Amer. Math. Soc.},
  volume={121},
  date={1966},
  pages={223--235},
  issn={0002-9947},
}

\bib{MR0799895}{article}{
  author={Dugas, Manfred},
  title={On reduced products of abelian groups},
  journal={Rend. Sem. Mat. Univ. Padova},
  volume={73},
  date={1985},
  pages={41--47},
  issn={0041-8994},
}

\bib{EilMoo}{book}{
  author={Eilenberg, Samuel},
  author={Moore, J. C.},
  title={Foundations of relative homological algebra},
  journal={Mem. Amer. Math. Soc.},
  volume={55},
  date={1965},
  pages={39},
  issn={0065-9266},
}

\bib{MR1753146}{book}{
  author={Enochs, Edgar E.},
  author={Jenda, Overtoun M. G.},
  title={Relative homological algebra},
  series={De Gruyter Expositions in Mathematics},
  volume={30},
  publisher={Walter de Gruyter \& Co., Berlin},
  date={2000},
  pages={xii+339},
  isbn={3-11-016633-X},
}

\bib{MattHandbook}{article}{
  author={Foreman, Matthew},
  title={Ideals and generic elementary embeddings},
  conference={ title={Handbook of set theory. Vols. 1, 2, 3}, },
  book={ publisher={Springer, Dordrecht}, },
  date={2010},
  pages={885--1147},
}

\bib{MR0647168}{article}{
  author={Gardner, B. J.},
  title={When are radical classes of abelian groups closed under direct products?},
  conference={ title={Algebraic structures and applications}, address={Nedlands}, date={1980}, },
  book={ series={Lect. Notes Pure Appl. Math.}, volume={74}, publisher={Dekker, New York}, },
  isbn={0-8247-1570-5},
  date={1982},
}

\bib{MR2985554}{book}{
  author={G\"{o}bel, R\"{u}diger},
  author={Trlifaj, Jan},
  title={Approximations and endomorphism algebras of modules. Volume 1},
  series={De Gruyter Expositions in Mathematics},
  volume={41},
  edition={Second revised and extended edition},
  publisher={Walter de Gruyter GmbH \& Co. KG, Berlin},
  date={2012},
  pages={xxviii+458},
  isbn={978-3-11-021810-7},
  isbn={978-3-11-021811-4},
}

\bib{MR3839274}{book}{
  author={Iacob, Alina},
  title={Gorenstein homological algebra},
  note={With a foreword by Sergio Estrada},
  publisher={CRC Press, Boca Raton, FL},
  date={2019},
  pages={xvii+214},
}

\bib{MR1994835}{book}{
  author={Kanamori, Akihiro},
  title={The higher infinite},
  series={Springer Monographs in Mathematics},
  edition={2},
  note={Large cardinals in set theory from their beginnings},
  publisher={Springer-Verlag},
  place={Berlin},
  date={2003},
  pages={xxii+536},
  isbn={3-540-00384-3},
  review={\MR {1994835 (2004f:03092)}},
}

\bib{PPT}{article}{
  title={Closure properties of $\varinjlim \mathcal C$},
  author={L.\ Positselski, Leonid},
  author={P\v {r}\'{i}hoda, Pavel},
  author={Trlifaj, Jan},
  journal={Journal of Algebra},
  volume={606},
  pages={30--103},
  year={2022},
}

\bib{MR2822215}{article}{
  author={Saor\'{\i }n, Manuel},
  author={\v {S}\v {t}ov\'{\i }\v {c}ek, Jan},
  title={On exact categories and applications to triangulated adjoints and model structures},
  journal={Adv. Math.},
  volume={228},
  date={2011},
  number={2},
  pages={968--1007},
  issn={0001-8708},
}

\bib{sar_trl_DeconAEC}{article}{
  title={Deconstructible abstract elementary classes of modules and categoricity},
  author={{\v {S}}aroch, Jan},
  author={Trlifaj, Jan},
  journal={arXiv preprint arXiv:2312.02623},
  year={2023},
}

\bib{MR0389953}{book}{
  author={Stenstr\"{o}m, Bo},
  title={Rings of quotients},
  series={Die Grundlehren der mathematischen Wissenschaften},
  volume={Band 217},
  note={An introduction to methods of ring theory},
  publisher={Springer-Verlag, New York-Heidelberg},
  date={1975},
  pages={viii+309},
}

\bib{MR0364497}{article}{
  author={Walker, Elbert A.},
  title={The groups $\sb {\beta }$},
  conference={ title={Symposia Mathematica, Vol. XIII}, address={Convegno di Gruppi Abeliani \& Convegno di Gruppi e loro Rappresentazioni, INDAM, Rome}, date={1972}, },
  book={ publisher={Academic Press, London-New York}, },
  date={1974},
  pages={245--255},
}



\begin{bibdiv}
\begin{biblist}
\bibselect{Bibliography}
\end{biblist}
\end{bibdiv}
\end{document}